\documentclass[11 pt]{amsart}
\usepackage{graphicx}
\usepackage{amsmath}
\usepackage{mathtools}
\usepackage
{amstext, amscd, setspace, tikz-cd,mathtools, enumerate, graphics, latexsym, siunitx}

\usepackage{pst-node}
\usepackage{tikz-cd}

\usepackage{amssymb}
\usepackage{hyperref}

\usepackage{PTSerif}
\usepackage{graphicx}

\usepackage[cmtip,all]{xy}
\newcommand{\properideal}{%
\mathrel{\ooalign{$\lneq$\cr\raise.22ex\hbox{$\lhd$}\cr}}}

\newcommand{\GL}{{\rm GL}}

\theoremstyle{definition}
\numberwithin{equation}{subsubsection}

\newtheorem{theorem}{Theorem}[section]

\newtheorem{lemma}[theorem]{Lemma}

\newtheorem{definition}[theorem]{Definition}

\newtheorem{hypothesis}[theorem]{Hypothesis}

\numberwithin{equation}{section}

\title[A Generalization of a Theorem of Nakajima-Landweber-Stong]{A Generalization of a Theorem of Nakajima-Landweber-Stong}

\author{Shubham Jaiswal, Tony J. Puthenpurakal}

\address{Department of Mathematics IIT Bombay, Powai, Mumbai 400 076, India.}

\email{mynameissj555@gmail.com, tputhen@gmail.com}


\date{\today}


\begin{document}

\begin{abstract}
The main result of this paper is a generalization of a
theorem of Nakajima-Landweber-Stong to the modular invariant rings of transvection groups over Dedekind domains. More precisely, let $A$ be a Dedekind domain and $K$ be its field of fractions. Assume that $A$ contains a finite field $\mathbb{F}_q$ with $q=p^r$ elements for a prime $p$. Let $n\geq 2$ and consider a finite subgroup $G$ of $\mathrm{GL}(A^n)$ such that every non-identity element of $G$ inside $\GL(K^n)$ is a transvection. Consider the ring $A[X_1,X_2,\dots, X_n]$ and let $G$ act linearly on the ring (fixing $A$). Then $(A[X_1,X_2,\dots, X_n])^{G}$ is regular.
\end{abstract}

\maketitle

\section{Introduction}

We begin by recalling the notion of transvections.

\begin{definition}[\S 8.2, \cite{smith1995polynomial}]

Let $K$ be a field and $V=K^n$. An element $\sigma\in \GL(K^n)$ is called a transvection with hyperplane $H_{\sigma}$, transvector $0\neq x\in H_{\sigma}$ and direction $\text{Span}_K\{x\}$ if there is a linear functional $\phi_{\sigma}:V\rightarrow K$ such that $H_{\sigma}=\ker(\phi_{\sigma})$ and $\sigma(v)=v+\phi_{\sigma}(v)\cdot x$ for all $v\in V$.
    
\end{definition}

We have the theorem by Nakajima-Landweber-Stong (Theorem 8.2.13 in \cite{smith1995polynomial}) which states: 

\begin{theorem}
    \label{NLS thm}

       Let $\mathbb{F}$ be a finite field of characteristic $p>0$. Let $n\geq 2$. We consider a finite subgroup $G$ of $\mathrm{GL}(\mathbb{F}^n)$ such that every non-identity element of $G$ inside $\GL(\mathbb{F}^n)$ is a transvection. Consider the ring $\mathbb{F}[X_1,X_2,\dots, X_n]$ and let $G$ act linearly on the ring (fixing $\mathbb{F}$). Then $(\mathbb{F}[X_1,X_2,\dots, X_n])^{G}$ is a polynomial algebra.  
\end{theorem}

In this paper we establish the following result which is a generalization of Theorem \ref{NLS thm} to the modular invariant rings of transvection groups over Dedekind domains.

\begin{theorem}\label{main trasvec}

    Let $A$ be a Dedekind domain and let $K$ be its field of fractions. Assume that $A$ contains a finite field $\mathbb{F}_q$ with $q=p^r$ elements for a prime $p$. Let $n\geq 2$. We consider a finite subgroup $G$ of $\mathrm{GL}(A^n)$ such that every non-identity element of $G$ inside $\GL(K^n)$ is a transvection. Consider the ring $A[X_1,X_2,\dots, X_n]$ and let $G$ act linearly on the ring (fixing $A$). Then $(A[X_1,X_2,\dots, X_n])^{G}$ is regular.    
\end{theorem}



    

 Firstly in Section \ref{aux sec}, we state and prove some integral lemmas. Then in Section \ref{main}, we prove Theorem \ref{main trasvec} by first proving it for fields, then for DVRs and finally for Dedekind domains.

\section{Integral Lemmas}\label{aux sec}

Consider the hypothesis below.

\begin{hypothesis}
    \label{hypo field}

Let $K$ be a field and assume that $K$ contains a finite field $\mathbb{F}_q$ with $q=p^r$ elements for a prime $p$. Let $D$ be a finite subgroup of $(K,+)$. Let $n\geq 2$. Let $E_{i,j}$ denote the $n \times n$ matrix with a $1$ in the $i^{th}$ row and $j^{th}$ column, and zeros elsewhere for $i, j \in \{1, \dots, n\}$. For $r = 1, \ldots, n-1$, consider a subgroup $E_{D}(r)$ of $\mathrm{GL}(K^n)$ which is generated by the set $\{S_i(\alpha)= I + \alpha E_{i,n} \mid \alpha \in D, i = 1, \ldots, r\}$ of transvections inside $\GL(K^n)$ and consider a subgroup $E^*_{D}(r)$ of $\mathrm{GL}(K^n)$ which is generated by the set $\{(S_i(\alpha))^T= I + \alpha E_{n,i} \mid \alpha \in D, i = 1, \ldots, r\}$ of transvections inside $\GL(K^n)$. 

\end{hypothesis}

\begin{lemma}\label{isom lemma} Assume Hypothesis \ref{hypo field}.

\begin{enumerate}
\item The subgroup $E_D(r)\subset \GL(K^n)$ is isomorphic to $D^r$ and thus an elementary abelian $p$-group.\smallskip

\item The subgroup $E^*_D(r)\subset \GL(K^n)$ is also isomorphic to $D^r$ and thus an elementary abelian $p$-group.
\end{enumerate}
    
\end{lemma}

\begin{proof}\hfill

\begin{enumerate}
\item 

We mimic the proof of Lemma 8.2.4 in \cite{smith1995polynomial}. By matrix multiplication 
\[
E_{i,j} E_{k,l} = \begin{cases}
0 & \text{for } j \neq k \\
E_{i,l} & \text{for } j = k.
\end{cases}
\]
Hence for $1 \leq i, j < n$,
\[
S_i(\alpha) S_j(\beta) = (I + \alpha E_{i,n})(I + \beta E_{j,n}) = I + \alpha E_{i,n} + \beta E_{j,n}
= (I + \beta E_{j,n})(I + \alpha E_{i,n}) = S_j(\beta) S_i(\alpha).
\]
Thus the matrices $S_i(\alpha)$, $S_j(\beta)$ where $i, j = 1, \ldots, n-1$ and $\alpha, \beta \in D$ commute pairwise so $E_{D}(r)$ is a commutative group. For $i = j$ we obtain $S_i(\alpha) S_i(\beta) = S_i(\alpha + \beta)$ and therefore the matrices $\{S_i(\alpha) \mid \alpha \in D\}$ for fixed $i$ form a subgroup of $\mathrm{GL}(n, K)$ isomorphic to $D$. The subgroup generated by $S_i(\alpha)$ and $S_j(\beta)$ for $i \neq j$ and $\alpha, \beta \in D$ have only the identity element in common. Therefore the map \[
\sigma_r : D^r \longrightarrow E_{D}(r)
\]
defined by $\sigma_r(\alpha_1, \ldots, \alpha_r) = S_1(\alpha_1) \cdots S_r(\alpha_r)$ is an isomorphism. Clearly $D$ is an elementary abelian $p$-group i.e. $D\cong (\mathbb{F}_p)^l$ for some $l$. So $E_{D}(r)$ is an elementary abelian $p$-group of order $p^{lr}$.\smallskip

\item For $1 \leq i, j < n$ and $\alpha, \beta \in D$ we have 
$(S_i(\alpha))^T (S_j(\beta))^T= (S_j(\beta) S_i(\alpha))^T=(S_i(\alpha) S_j(\beta))^T= (S_j(\beta))^T (S_i(\alpha))^T$. Thus $E^*_{D}(r)$ is a commutative group. Also for $i = j$ we obtain that $(S_i(\alpha))^T (S_i(\beta))^T=(S_i(\beta)S_i(\alpha))^T = (S_i(\beta + \alpha))^T=(S_i(\alpha+\beta))^T$. Therefore the map
\[
\lambda_r : E^*_D(r) \longrightarrow E_{D}(r)
\]
which maps $\lambda_r((S_1(\alpha_1))^T \cdots (S_r(\alpha_r))^T) = S_1(\alpha_1) \cdots S_r(\alpha_r)$ is an isomorphism. Hence we are done by part (1).
\end{enumerate}\end{proof}




   





\begin{hypothesis}
\label{hypo dvr}
  Let $(\mathcal{O},(\pi))$ be a DVR and let $K=\mathcal{O}_\pi$ where $\mathcal{O}_\pi$ is the localization of $\mathcal{O}$ at the multiplicative set $\{\pi^i\}_{i}$ (which is the field of fractions of $\mathcal{O}$). Assume that $\mathcal{O}$ contains a finite field $\mathbb{F}_q$ with $q=p^r$ elements for a prime $p$. Let $D$ be a finite subgroup of $(\mathcal{O},+)$. Let $n\geq 2$. For $r = 1, \ldots, n-1$, consider a subgroup $E_{D}(r)$ of $\mathrm{GL}(\mathcal{O}^n)$ which is generated by the set $\{S_i(\alpha)= I + \alpha E_{i,n} \mid \alpha \in D, i = 1, \ldots, r\}$ of transvections inside $\GL(K^n)$ and consider a subgroup $E^*_{D}(r)$ of $\mathrm{GL}(\mathcal{O}^n)$ which is generated by the set $\{(S_i(\alpha))^T= I + \alpha E_{n,i} \mid \alpha \in D, i = 1, \ldots, r\}$ of transvections inside $\GL(K^n)$. 
  \end{hypothesis}

    
  






\begin{lemma}\label{new lemma}

   Assume Hypothesis \ref{hypo dvr}. Consider a $K$-basis $\{v_1,v_2,\dots v_n\}$ of $K^n$.
   
\begin{enumerate}
\item \begin{enumerate}

\item  Suppose that $\sigma\in \GL(\mathcal{O}^n)\subset \GL(K^n)$ with $ord(\sigma)=p$ such that $\sigma(v_i)=v_i$ for all $1\leq i\leq n-1$ and $\sigma(v_n)=v_{n-1}+v_n$. Then there exists an $\mathcal{O}$-basis $\{w_1,w_2,\dots w_n\}$ of $\mathcal{O}^n$ with $\sigma(w_i)=w_i$ for $1\leq i\leq n-1$ and $\sigma(w_n)= w_{n-1}+ w_n$.\smallskip

\item    Further suppose that $\sigma'\in \GL(\mathcal{O}^n)\subset \GL(K^n)$ such that $\sigma'(v_i)=v_i$ for all $1\leq i\leq n-1$ and $\sigma'(v_n)=\alpha v_{j}+v_n$ for some $1\leq j\leq n-1$ and some $\alpha\in \mathcal{O}$. Then for the same $\mathcal{O}$-basis $\{w_1,w_2,\dots w_n\}$ of $\mathcal{O}^n$ above, we have $\sigma'(w_i)=w_i$ for $1\leq i\leq n-1$ and $\sigma'(w_n)= \alpha w_{j}+ w_n$.
\end{enumerate} 

\smallskip

\item \begin{enumerate}

\item Suppose that $\sigma\in \GL(\mathcal{O}^n)\subset \GL(K^n)$ with $ord(\sigma)=p$ such that $\sigma(v_i)=v_i$ for all $1\leq i\leq n-2$ and $i=n$ and $\sigma(v_{n-1})=v_{n-1}+v_n$. Then there exists an $\mathcal{O}$-basis $\{w'_1,w'_2,\dots w'_n\}$ of $\mathcal{O}^n$ with $\sigma(w'_i)=w'_i$ for $1\leq i\leq n-2$ and $i=n$ and $\sigma(w'_{n-1})= w'_{n-1}+ w'_n$.\smallskip

\item    Further suppose that $\sigma'\in \GL(\mathcal{O}^n)\subset \GL(K^n)$ such that $\sigma'(v_i)=v_i$ for all $i\neq j$ and $\sigma'(v_j)= v_{j}+\alpha v_n$ for some $1\leq j\leq n-1$ and some $\alpha\in \mathcal{O}$. Then for the same $\mathcal{O}$-basis $\{w'_1,w'_2,\dots w'_n\}$ of $\mathcal{O}^n$ above, we have $\sigma'(w'_i)=w'_i$ for all $i\neq j$ and $\sigma'(w'_j)= w'_{j}+\alpha w'_n$.

   \end{enumerate}

  \end{enumerate}
\end{lemma}

\begin{proof}
\hfill

\begin{enumerate}
\item \begin{enumerate}
    \item Let $\{v_1,v_2,\dots , v_n\}$ be the given $K$-basis of $K^n$ such that $\sigma(v_i)=v_i$ for $1\leq i\leq n-1$ and $\sigma(v_n)=v_{n-1} + v_n$. We can assume $v_i\in \mathcal{O}^n\backslash \pi\mathcal{O}^n$. We will show that for a suitable $s\in K$, $\{v_1,v_2,\dots , v_{n-1}, v_n+sv_{n-1}\}$ is the required $\mathcal{O}$-basis. We prove this by induction on $n\geq 2$.\smallskip

Let $n=2$ and let $\{v_1,v_2\}$ be a $K$-basis of $K^2$ such that $\sigma(v_1)=v_1$ and $\sigma(v_2)=v_1 + v_2$. We can assume $v_1,v_2\in \mathcal{O}^2\backslash \pi\mathcal{O}^2$. Let $w_1=v_1$. Let $\{w_1,w_2\}$ be an $\mathcal{O}$-basis of $\mathcal{O}^2$. Let $w_2=av_1 +bv_2$ with $a,b\in K$. Note that $b\neq 0$. Since $(\mathcal{O},\pi)$ is a DVR and $K=\mathcal{O}_{\pi}$, we have $b=\pi^k b'$ where $k\in \mathbb{Z}$ and $b'\in \mathcal{O}^{\times}$. Now $\sigma(w_2)=av_1 +b(v_1+v_2)=w_2+bw_1$. Since $w_2\in \mathcal{O}^2$, we have $\sigma(w_2)-w_2=bw_1\in \mathcal{O}^2$. So $\pi^kb'w_1\in \mathcal{O}^2$. Suppose $k<0$, then $b'w_1\in \pi^{-k}\mathcal{O}^2$. Thus $w_1\in \pi\mathcal{O}^2$ which contradicts our assumption. Thus $k\geq 0$. So $b\in \mathcal{O}$.\smallskip

Let $v_2 =c w_1 +d w_2$ with $c,d\in \mathcal{O}$. As $\sigma(v_2)=v_1+v_2$, we have $cw_1+d(w_2+bw_1)=w_1 + cw_1 +d w_2$. Thus $(db-1)w_1=0$. Hence $db=1$. As we have $b,d\in \mathcal{O}$, so $b\in \mathcal{O}^{\times}$. Thus $\{w_1,b^{-1}w_2\}=\{v_1,v_2 +b^{-1}av_1\}$ is the required $\mathcal{O}$-basis for $\mathcal{O}^2$ with $s=b^{-1}a\in K$.

\smallskip

Now let $n\geq 3$ Let $w_1=v_1$. We have the following exact sequences. $$0\rightarrow F=Kv_1 \rightarrow V=K^n=\oplus_{i=1}^n\ Kv_i \rightarrow \bar{V}=K^{n-1}=\oplus_{i=2}^n\ Kv_i\rightarrow 0$$ $$\text{and}\ 0\rightarrow E=\mathcal{O}w_1 \rightarrow W=\mathcal{O}^n\rightarrow \bar{W}=\mathcal{O}^{n-1}\rightarrow 0$$

Now $\sigma|_{\bar{W}}\in \GL(\bar{W})\subset \GL(\bar{V})$. Also $\{\bar{v_2},\bar{v_3},\dots \bar{v_n}\}$ is a $K$-basis of $\bar{V}$ such that $\sigma(\bar{v_i})=\bar{v_i}$ for all $2\leq i\leq n-1$ and $\sigma(\bar{v_n})=\bar{v_{n-1}}+\bar{v_n}$. By induction we have an $\mathcal{O}$-basis\\
$\{\bar{v_2}, \bar{v_3}, \dots, \bar{v_{n-1}},\bar{v_n}+s\bar{v_{n-1}}\}$ of $\bar{W}$ for some $s\in K$. Now $\{v_1,v_2 +a_2 v_1, \dots, v_{n-1}+a_{n-1}v_1, v_n+sv_{n-1} +a_n v_1\}$ is an $\mathcal{O}$-basis of $W$ where $a_i\in \mathcal{O}$ for all $2\leq i\leq n$. Thus $\{v_1,v_2,\dots , v_{n-1}, v_n+sv_{n-1}\}$ is the required $\mathcal{O}$-basis.\smallskip

\item  The assertion follows as $\sigma'(v_{n}+sv_{n-1})=\sigma'(v_n)+s\sigma'(v_{n-1})=\alpha v_j +(v_n + s v_{n-1})$.

\end{enumerate}
\medskip

\item \begin{enumerate}
    \item  Let $\{v_1,v_2,\dots , v_n\}$ be the given $K$-basis of $K^n$ such that $\sigma(v_i)=v_i$ for all $1\leq i\leq n-2$ and $i=n$ and $\sigma(v_{n-1})=v_{n-1}+v_n$. We can assume $v_i\in \mathcal{O}^n\backslash \pi\mathcal{O}^n$. Let $v'_i=v_i$ for $1\leq i\leq n-2$ and $v'_{n-1}=v_n$ and $v'_{n}=v_{n-1}$. Thus $\{v'_1,v'_2,\dots , v'_n\}$ is a $K$-basis of $K^n$ such that $\sigma(v'_i)=v'_i$ for all $1\leq i\leq n-1$ and $\sigma(v'_{n})=v'_{n-1}+v'_n$. By proof of part (1), we have that for a suitable $s\in K$, $\{v'_1,v'_2,\dots , v'_{n-1}, v'_n+sv'_{n-1}\}$ is an $\mathcal{O}$-basis of $\mathcal{O}^n$. Let $w'_i=v'_i=v_i$ for $1\leq i\leq n-2$ and $w'_{n-1}=v'_n+sv'_{n-1}=v_{n-1}+sv_n$ and $w'_{n}=v'_{n-1}=v_n$. Thus $\{w'_1,w'_2,\dots , w'_{n}\}$ is the required $\mathcal{O}$-basis as $\sigma(w'_{n-1})=\sigma(v_{n-1}+sv_n)=\sigma(v_{n-1})+s\sigma(v_n)=v_{n-1}+v_n +sv_n=v_{n-1}+sv_n +v_n=w'_{n-1} +w'_n$.\smallskip

\item  First suppose $j=n-1$. Thus $\sigma'(v_i)=v_i$ for all $1\leq i\leq n-2$ and $i=n$ and $\sigma'(v_{n-1})= v_{n-1}+\alpha v_n$. Hence $\sigma'(w'_i)=w'_i$ for all $1\leq i\leq n-2$ and $i=n$ and $\sigma'(w'_{n-1})=\sigma'(v_{n-1}+sv_n)=v_{n-1}+\alpha v_n +sv_n=w'_{n-1}+\alpha w'_n$. Now suppose $1\leq j\leq n-2$. Thus $\sigma'(v_i)=v_i$ for all $i\neq j$ and $\sigma'(v_{j})= v_{j}+\alpha v_n$. In particular $\sigma'(v_{n-1})=v_{n-1}$. Hence $\sigma'(w'_{n-1})=\sigma'(v_{n-1}+sv_n)=v_{n-1}+sv_n=w'_{n-1}$. Thus $\sigma'(w'_i)=w'_i$ for all $i\neq j$. Finally we have $\sigma'(w'_j)=\sigma'(v_j)=v_j+\alpha v_n=w'_j +\alpha w'_n$.
\end{enumerate}

\end{enumerate}
\end{proof}




\begin{lemma} Assume Hypothesis \ref{hypo dvr}.

\begin{enumerate}
\item The subgroup $E_D(r)\subset \GL(\mathcal{O}^n)$ is isomorphic to $D^r$ and thus an elementary abelian $p$-group.\smallskip

\item The subgroup $E^*_D(r)\subset \GL(\mathcal{O}^n)$ is isomorphic to $D^r$ and thus an elementary abelian $p$-group.

\end{enumerate}
\end{lemma}

\begin{proof}\hfill

\begin{enumerate}
\item By Lemma \ref{new lemma} (1), $E_{D}(r)$ is generated by the same set $\{S_i(\alpha)= I + \alpha E_{i,n} \mid \alpha \in D, i = 1, \ldots, r\}$ inside $\GL(\mathcal{O}^n)$. Thus the result follows by proof of Lemma \ref{isom lemma} (1).\smallskip

\item  By Lemma \ref{new lemma} (2), $E^*_{D}(r)$ is generated by the same set $\{(S_i(\alpha))^T= I + \alpha E_{n,i} \mid \alpha \in D, i = 1, \ldots, r\}$ inside $\GL(\mathcal{O}^n)$. Thus the result follows by proof of Lemma \ref{isom lemma} (2).

\end{enumerate}\end{proof}

 We provide the proof of the following lemma for the sake of completeness.
 
\begin{lemma}[Lemmas 8.2.7, 8.2.8, 8.2.9, \cite{smith1995polynomial}] 
\label{trans lem}

Let $K$ be a field and let $S,T\in \GL(K^n)$ be transvections. Let $V=K^n$.

\begin{enumerate}
\item Suppose $S,T$ have the same hyperplane. Then $S$ and $T$ commute and $ST$ is identity or a transvection with the same hyperplane and transvector $x+y$, where $x$ is a transvector for $S$ and $y$ is a transvector for $T$.

\smallskip

\item Suppose $S,T$ have the same direction. Then $S$ and $T$ commute and $ST$ is identity or a transvection with the same direction.

\smallskip

\item  Suppose $ST$ is also a transvection. Then $S$ and $T$ have either the same hyperplane or the same direction.

\end{enumerate}
    
\end{lemma}

\begin{proof}\hfill

    \begin{enumerate}
\item  Let $H$ be the common hyperplane. We have linear functionals $\phi_S,\phi_T : V \rightarrow K$ such that $H=\ker(\phi_{S})=\ker(\phi_{T})$ and $S(v)=v+\phi_{S}(v)\cdot x_S$ and $T(v)=v+\phi_{T}(v)\cdot y_T$ for all $v\in V$ where $x_S$ and $y_T$ are respective transvectors for $S$ and $T$. Observe that $\phi_T=c\cdot \phi_S$ for some $c\in K$. We redefine the transvectors for $S$ and $T$ as $x=x_S$ and $y=c\cdot y_T$. Let $\phi=\phi_S$. Thus for all $v\in V$, $S(v)=v+\phi(v)\cdot x$ and $T(v)=v+\phi(v)\cdot y$. As $y\in H$, so $S(y)=y$. Thus $ST(v)=S(v+\phi(v)\cdot y)=S(v)+\phi(v)S(y)=(v+\phi(v)\cdot x)+ \phi(v)\cdot y=v +\phi(v)\cdot (x+y)=v +\phi(v)\cdot (y+x)=TS(v)$. Also $ST=I$ when $x+y=0$.

\smallskip

\item Let $x$ be the common transvector. We have linear functionals $\phi_S,\phi_T : V \rightarrow K$ such that $H_S=\ker(\phi_{S}),\ H_T=\ker(\phi_{T})$ and $x\in H_S\cap H_T$ and $S(v)=v+\phi_{S}(v)\cdot x$ and $T(v)=v+\phi_{T}(v)\cdot x$ for all $v\in V$. As $x\in H_S$, so $S(x)=x$. Thus $ST(v)=S(v+\phi_T(v)\cdot x)=S(v)+\phi_T(v)S(x)=(v+\phi_S(v)\cdot x)+ \phi_T(v)\cdot x=v +(\phi_S(v)+\phi_T(v))\cdot x=v +(\phi_T(v)+\phi_S(v))\cdot x=TS(v)$. 

\smallskip

\item  We have linear functionals $\phi_S,\phi_T : V \rightarrow K$ such that $H_S=\ker(\phi_{S}),\ H_T=\ker(\phi_{T})$ and $x$ and $y$ are respective transvectors for $S$ and $T$ and $S(v)=v+\phi_{S}(v)\cdot x$ and $T(v)=v+\phi_{T}(v)\cdot y$ for all $v\in V$. Assume $H_S\neq H_T$. Thus $H_S\not \subset H_T$ and $H_T\not \subset H_S$. So we can choose $A\in H_S\backslash H_T$ and $B\in H_T\backslash H_S$. Thus $\phi_T(A)\neq 0\neq \phi_S(B)$. We can scale $A$ and $B$ and assume $\phi_T(A)=1=\phi_S(B)$. Also $\phi_S(A)=0=\phi_T(B)$. Hence $S(B)=B+x,\ S(A)=A,\ T(A)=A+y,\ T(B)=B$. Therefore $ST(A)=S(A+y)=A+S(y)=A+y+\phi_S(y)\cdot x$ and $ST(B)=S(B)=B+x$. Thus $\text{Span}_K\{x,y\}\subset \text{Im} (ST-1)$. As $ST$ is a transvection, $\text{Im}(ST-1)$ is one-dimensional. Thus $x$ and $y$ are linearly dependent over $K$. Hence $S$ and $T$ have same direction.

\end{enumerate}\end{proof}

We prove an important lemma below as a generalization of Propositions 8.2.10, 8.2.11 in \cite{smith1995polynomial}.

\begin{lemma}\label{conjugate lem}
    Let $K$ be a field and assume that $K$ contains a finite field $\mathbb{F}_q$ with $q=p^r$ elements for a prime $p$. Let $n\geq 2$. We consider a finite subgroup $G$ of $\mathrm{GL}(K^n)$ such that every non-identity element of $G$ is a transvection. Then $G$ is an elementary abelian $p$-group which is conjugate to a subgroup of $E_D(r)$ or $E^*_D(r)$ for some finite subgroup $D$ of $(K,+)$ and for some $r\in \{1,\dots, n-1\}$. 
\end{lemma}

\begin{proof}
All elements of $G$ except identity have order $p$. If $S,T\in G$ then we have that either $S$ and $T$ have the same hyperplane or they have the same direction by Lemma \ref{trans lem} (3). Thus $S$ and $T$ commute by Lemma \ref{trans lem} (1), (2). Thus $G$ is abelian. As all non-identity elements of $G$ have order $p$, so $G$ is an elementary abelian $p$-group. Now we claim that all elements of $G$ have either the same hyperplane or the same direction. Let $T,U\in G$ with different directions. Then by Lemma \ref{trans lem} (3), $T$ and $U$ have the same hyperplane. Suppose $S\in G$. Then the direction of $S$ must be different from either the direction of $T$ or the direction of $U$. By Lemma \ref{trans lem} (3), the hyperplane of $S$ must be the common hyperplane of $T$ and $U$. Thus all elements of $G$ have the same hyperplane. Similarly if we have $T,U\in G$ with different hyperplanes, then one can show that all elements of $G$ have the same direction.\smallskip

Now suppose all elements of $G$ have the same hyperplane $H$. By proof of Lemma \ref{trans lem} (1), we have a linear functional $\phi:V=K^n\rightarrow K$ such that $H=\ker(\phi)$ and for each $S\in G$ we have $S(v)=v+\phi(v)\cdot x_S$ where $x_S$ is a transvector for $S$. Let $\{v_1,\dots,v_{n-1}\}$ be a $K$-basis for $H$. Extend this to a $K$-basis $\{v_1,\dots,v_{n-1}, v_n\}$ of $V$. Now $\phi(v_i)=0$ and $S(v_i)=v_i$ for $1\leq i\leq n-1$ and $\phi(v_n)\neq 0$. Without loss of generality we can assume that $\phi(v_n)=1$ and thus $S(v_n)=v_n+x_S$. For each $S\in G$, we have $x_S\in H$. So $x_S=a_{S1}v_1+\cdots +a_{Sn-1}v_{n-1}$ for $a_{Si}\in K$ for all $1\leq i\leq n-1$. Let $D$ be subgroup of $(K,+)$ generated by $\{a_{S1},\dots, a_{Sn-1}\}_{S\in G}$. As $G$ is finite, so is $D$. As $S(v_n)=a_{S1}v_1+\cdots +a_{Sn-1}v_{n-1}+v_n$, so the matrix of $S$ with respect to the basis $\{v_1,\dots, v_n\}$ is $S_1(a_{S1})\cdots S_{n-1}(a_{Sn-1})$ where $S_i(\alpha)$ are as in the proof of Lemma \ref{isom lemma}. Thus by change of basis matrix, $G$ is conjugate to a subgroup of $E_D(r)$ for some $r\in \{1,\dots, n-1\}$.\smallskip

Now suppose all elements of $G$ have the same direction. Let $x$ be the common transvector. For each $S\in G$, we have linear functional $\phi_S: V=K^n \rightarrow K$ such that $H_S=\ker(\phi_{S})$ and $S(v)=v+\phi_{S}(v)\cdot x$ and for all $v\in V$. We extend $\{x\}$ to a $K$-basis $\{v_1,\dots, v_{n-1}, v_n=x\}$ of $V$. As $x\in H_S$, so $\phi_S(x)=0$ and $S(x)=x$. Also for $1\leq i\leq n-1$, we have $S(v_i)=v_i+\phi_S(v_i)\cdot v_n$. Let $D$ be subgroup of $(K,+)$ generated by $\{\phi_S(v_1),\dots, \phi_S(v_{n-1})\}_{S\in G}$. As $G$ is finite, so is $D$. Also the matrix of $S$ with respect to the basis $\{v_1,\dots , v_n\}$ is $(S_1(\phi_S(v_1)))^T\cdots (S_{n-1}(\phi_S(v_{n-1})))^T$ where $(S_i(\alpha))^T$ are as in the proof of Lemma \ref{isom lemma}. Thus by change of basis matrix, $G$ is conjugate to a subgroup of $E^*_D(r)$ for some $r\in \{1,\dots, n-1\}$.\end{proof}

The following is a generalization of Lemma 8.2.14 in \cite{smith1995polynomial}.

\begin{lemma}\label{comm alg lemma}
     Let $K$ be a field and assume that $K$ contains a finite field $\mathbb{F}_q$ with $q=p^r$ elements for a prime $p$. Let $D$ be a finite subgroup of $(K,+)$ and let $E$ be a group isomorphic to $D^r$ for some $r\in \mathbb{N}$. Let $A\subset E$ be an additive subgroup. Then there exists an isomorphism $\lambda: E\rightarrow D^r$ such that $\lambda (A)= A_1\oplus \cdots \oplus A_r$ where $A_i\subset D$ are subgroups for all $1\leq i\leq r$.
\end{lemma}

\begin{proof}
We prove the result by induction on $r$. For $r=1$, the result is trivial. Assume the result established for all $A\subset E$ where $E$ is isomorphic to $D^r$. Now let $E$ be isomorphic to $D^{r+1}$ and let $A\subset E$ be an additive subgroup. Let $E'\subset E$ such that $E'$ is isomorphic to $D^r$. Clearly $D$ is an elementary $p$-group and so is $E$. Thus they are $\mathbb{F}_p$-vector spaces. Consider the diagram

 \[
\begin{tikzcd}
 0 \arrow[r] & A\cap E' \arrow[d, "\phi'"] \arrow[r, "\alpha'"] & A \arrow[d, "\phi"] \arrow[r, "\alpha''"] & A'' \arrow[d, "\phi''"] \arrow[r] & 0 \\
  0 \arrow[r] & E' \arrow[r, "\beta'"] & E \arrow[r, "\beta''"] & E'' \ar[r] & 0
\end{tikzcd}
\]

where both rows are exact and $\phi',\phi$ are monomorphisms. We will show that even $\phi''$ is a monomorphism. Let $a''\in \ker (\phi'')$. Let $a\in A$ such that $\alpha''(a)=a''$. Then $\beta''(\phi(a))=\phi''(\alpha''(a))=\phi''(a'')=0$. Thus $\phi(a)\in \ker (\beta'')=Im(\beta')$. Thus there is a unique $e'\in E'$ such that $\beta'(e')=\phi(a)$. As $\phi$ and $\beta'$ are inclusion maps, we have $a=e'\in A\cap E'$. Thus $a\in Im(\alpha')=ker(\alpha'')$ and so $a''=\alpha''(a)=0$.\smallskip

Now both rows in the diagram are exact sequences of $\mathbb{F}_p$-vector spaces and hence split. Thus $A\cong (A\cap E')\oplus A''$ and $E\cong E'\oplus E''$. Thus $E''\cong D$. Also since $\phi''$ is a monomorphism, we can view $A''$ as a subgroup of $E''$. Consider an isomorphism $\nu :E''\rightarrow D$ and let $\nu(A'')=A_{r+1}$. 
By induction, we have an isomorphism $\lambda:E'\rightarrow D^r$ such that $\lambda (A\cap E')=A_1\oplus \cdots \oplus A_r$. Since $E$ is isomorphic to $D^{r+1}$ and $E'$ is an $\mathbb{F}_p$-vector subspace of $E$, we can identify $E$ with $E'\oplus E''$ and extend the isomorphism $\lambda$ to an isomorphism $\tilde{\lambda}: E= E'\oplus E''\rightarrow D^{r+1}=D^r\oplus D$ by defining $\tilde{\lambda}(e'+e'')=\lambda(e')+\nu(e'')$. Then clearly $\tilde{\lambda}(A)=\tilde{\lambda}((A\cap E')\oplus A'')=\lambda(A\cap E')\oplus \nu (A'')=A_1\oplus \cdots \oplus A_r\oplus A_{r+1}$.\end{proof}

\section{Proof of Theorem \ref{main trasvec}}\label{main}

\subsection{Theorem \ref{main trasvec} for fields}

\begin{theorem}\label{main for field}
     Let $K$ be a field and assume that $A$ contains a finite field $\mathbb{F}_q$ with $q=p^r$ elements for a prime $p$. Let $n\geq 2$. We consider a finite subgroup $G$ of $\mathrm{GL}(K^n)$ such that every non-identity element of $G$ inside $\GL(K^n)$ is a transvection. Consider the ring $K[X_1,X_2,\dots, X_n]$ and let $G$ act linearly on the ring (fixing $K$). Then $(K[X_1,X_2,\dots, X_n])^{G}$ is a polynomial algebra.  
\end{theorem}

\begin{proof}
    By Lemma \ref{conjugate lem}, $G$ is an elementary abelian $p$-group which is conjugate to a subgroup of $E_D(r)$ or $E^*_D(r)$ for some finite subgroup $D$ of $(K,+)$ and for some $r\in \{1,\dots, n-1\}$. Suppose $G$ is conjugate to a subgroup of $E_D(r)$. We can view $G$ as a subgroup of $E_D(r)$. Then by Lemma \ref{isom lemma} (1) and by Lemma \ref{comm alg lemma}, we may choose subgroups $A_1,\dots, A_r\subset D$ so that $G=\{I + \alpha_1 E_{1,n}+\cdots +\alpha_r E_{r,n} \mid \alpha_i \in A_i, i = 1, \ldots, r\}$. Let $\{y_1, y_2,\dots, y_n\}$ be standard dual $K$-basis of ${(K^n)}^*$. With respect to this basis, the action of $G$ is given by \[
(I + \alpha_1 E_{1,n}+\cdots +\alpha_r E_{r,n})\left(y_{j}\right)= \begin{cases}y_{j} & j \geq  r+1 \\ y_{j}+\alpha_j y_{n} & j\leq r .\end{cases}
\] Let \[
g_{i}=\prod_{\alpha_i \in A_i}\left(y_{i}+\alpha_i y_{n}\right)\]
for $i=1, \ldots, r$. Now $g_i\in K[y_1,\dots, y_n]^{G}$. By Proposition 5.3.7 in \cite{smith1995polynomial}, $g_1,\dots, g_r, y_{r+1}, \dots , y_n$ are a system of parameters. Since $deg(g_i)=|A_i|$ for all $1\leq i\leq r$, we have\\
$deg(g_1)\cdots deg(g_r)deg(y_{r+1})\cdots deg(y_{n})=|A_1|\cdots|A_r|=|G|$. Thus by Proposition 5.5.5 in \cite{smith1995polynomial}, \[
K\left[y_{1}, \ldots, y_{n}\right]^{G}=K\left[g_{1}, \ldots, g_{r}, y_{r+1}, \ldots, y_{n}\right]
.\]

Now suppose $G$ is conjugate to a subgroup of $E^*_D(r)$. We can view $G$ as a subgroup of $E^*_D(r)$. Then by Lemma \ref{isom lemma} (2) and by Lemma \ref{comm alg lemma}, we may choose subgroups $A_1,\dots, A_r\subset D$ so that $G=\{I + \alpha_1 E_{n,1}+\cdots +\alpha_r E_{n,r} \mid \alpha_i \in A_i, i = 1, \ldots, r\}$. Let $\{y_1, y_2,\dots, y_n\}$ be standard dual $K$-basis of ${(K^n)}^*$. With respect to the this basis, the action of $G$ is given by
\[
(I + \alpha_1 E_{n,1}+\cdots +\alpha_r E_{n,r})\left(y_{j}\right)= \begin{cases}y_{j} & j \neq n \\ y_{n}+\alpha_1 y_1+\cdots +\alpha_r y_r  & j=n .\end{cases}
\] Let
\[
g=\prod_{\alpha_1\in A_1,\dots, \alpha_r \in A_r}\left(y_{n}+\alpha_r y_r + \cdots +\alpha_1y_1\right).\] Now $g\in K[y_1,\dots, y_n]^{G}$. By Proposition 5.3.7 in \cite{smith1995polynomial}, $y_1,\dots, y_{n-1}, g$ are a system of parameters. Since $deg(g)=|A_1|\cdots |A_r|$, we have $deg(y_1)\cdots deg(y_{n-1})deg(g)=|G|$. Hence by Proposition 5.5.5 in \cite{smith1995polynomial}, \[
K\left[y_{1}, \ldots, y_{n}\right]^{G}=K\left[y_{1}, \ldots, y_{n-1}, g\right]
\]

\end{proof}

\subsection{Theorem \ref{main trasvec} for DVRs}

\begin{theorem}\label{main for DVR}
     Let $(\mathcal{O},(\pi))$ be a DVR and let $K=\mathcal{O}_\pi$ where $\mathcal{O}_\pi$ is the localization of $\mathcal{O}$ at the multiplicative set $\{\pi^i\}_{i}$ (which is the field of fractions of $\mathcal{O}$). Assume that $\mathcal{O}$ contains a finite field $\mathbb{F}_q$ with $q=p^r$ elements for a prime $p$.  Let $n\geq 2$. We consider a finite subgroup $G$ of $\mathrm{GL}(\mathcal{O}^n)$ such that every non-identity element of $G$ inside $\GL(K^n)$ is a transvection. Consider the ring $\mathcal{O}[X_1,X_2,\dots, X_n]$ and let $G$ act linearly on the ring (fixing $\mathcal{O}$). Then $(\mathcal{O}[X_1,X_2,\dots, X_n])^{G}$ is a polynomial algebra. 
\end{theorem}

\begin{proof}

 By Lemma \ref{conjugate lem}, $G$ inside $\GL(K^n)$ is an elementary abelian $p$-group which is conjugate to a subgroup of $E_D(r)$ or $E^*_D(r)$ for some finite subgroup $D$ of $(K,+)$ and for some $r\in \{1,\dots, n-1\}$. We can view $G$ as a subgroup of $E_D(r)$ or $E^*_D(r)$. Let $D'=\pi^k D$ where $k\in \mathbb{Z}$ such that $D'\subset (\mathcal{O},+)$. Let $U=diag(1,\dots,1, \pi^k)$. Now $US_i(\pi^k\alpha)U^{-1}=S_i(\alpha)$ and ${(U^T)}^{-1} (S_i(\pi^k\alpha))^T U^T =(S_i(\alpha))^T$ for $\alpha\in D$. Thus $E_{D'}(r)$ is conjugate to $E_D(r)$ inside $\GL(K^n)$ and $E^*_{D'}(r)$ is conjugate to $E^*_D(r)$ inside $\GL(K^n)$. Hence without loss of generality, we can assume that $D$ is a subgroup of $(\mathcal{O},+)$ and $G$ is a subgroup of $E_D(r)$ or $E^*_D(r)$.\smallskip

 Suppose $G$ is a subgroup of $E_D(r)$. Then by Lemma \ref{isom lemma} (1) and by Lemma \ref{comm alg lemma}, we may choose subgroups $A_1,\dots, A_r\subset D$ so that inside $\GL(K^n)$ we have $G=\{I + \alpha_1 E_{1,n}+\cdots +\alpha_r E_{r,n}=S_1(\alpha_1)\cdots S_r(\alpha_r) \mid \alpha_i \in A_i, i = 1, \ldots, r\}$. By Lemma \ref{new lemma} (1), $G$ is the same set inside $\GL(\mathcal{O}^n)$. Let $\{z_1, z_2,\dots, z_n\}$ be dual $\mathcal{O}$-basis of ${(\mathcal{O}^n)}^*$ corresponding to the $\mathcal{O}$-basis of $\mathcal{O}^n$ constructed in Lemma \ref{new lemma} (1). With respect to this basis, the action of $G$ is given by
\[
I + \alpha_1 E_{1,n}+\cdots +\alpha_r E_{r,n}\left(z_{j}\right)= \begin{cases}z_{j} & j \geq r+1 \\ z_{j}+\alpha_j z_{n} & j\leq r .\end{cases}
\]
Let 
\[
f_{i}=\prod_{\alpha_i\in A_i}\left(z_{i}+\alpha_iz_{n}\right)\]
for $i=1, \ldots, r$. Let $R=\mathcal{O}[z_1,\dots,z_n]$, $S=\mathcal{O}[z_1,\dots,z_n]^{G}$ and $U=\mathcal{O}[f_1,\dots,f_r,z_{r+1},\dots,z_r]$. Hence $f_{i} \in S$ and each $f_i$ is monic in $z_i$. Now $R/(f_1,\dots,f_r,z_{r+1},\dots, z_n)\cong \mathcal{O}[z_1,\dots, z_r]/(z_1^{|A_1|},\dots,z_r^{|A_r|})$ is finitely generated $<\bar{1},\bar{u_1},\dots, \bar{u_s}>$ as an $\mathcal{O}$-module where $deg(u_i)\geq 1$. We claim that $R$ is finitely generated over $U$ as $<1, u_1,\dots,u_s>$. Let $N=<1, u_1,\dots,u_s>$. It is enough to show that $N_m=R_m$ for all $m\geq 0$. We prove this by induction. Clearly it holds for $m=0$. Assume by induction for all $d< m$. Let $\alpha\in R_m$. Hence $\alpha=\Sigma u_j \xi_j+ \delta$ where $\delta\in (f_1,\dots,f_r,z_{r+1},\dots, z_n)R$. So $\alpha= \Sigma u_j\xi_j +\Sigma f_i\delta_i +\Sigma z_k \delta'_k$. Now all $\delta_i\in R_d$ and $\delta'_k\in R_{d'}$ for $d,d'<m$. Thus by induction $\delta_i\in N_d$ and $\delta'_k\in N_{d'}$. Hence $\alpha\in N_n$ which proves our claim.

Let $K'$, $Q$ and $T$ be respective field of fractions of $R$, $S$ and $U$. Now $K'/Q$ is Galois with Galois group $G$. Also $[K:T]\leq deg(f_1)\cdots deg(f_r)=|A_1|\cdots |A_r|=|G|$. Thus $Q=T$. Now $R$ is finite integral over $U$. Thus $R/\pi R\cong \kappa [z_1,\dots,z_n]$   is finite integral over $U/\pi U\cong \kappa [f_1,\dots,f_r,z_{r+1},\dots,z_n]$ where $\kappa =\mathcal{O}/\pi \mathcal{O}$. Hence $U/\pi U$ has dimension $n$, thus is regular. Let $\mathfrak{M}$ be the maximal homogenous ideal of $U$. Now $\pi\not \in \mathfrak{M}^2$. As $U/\pi U$ is regular, thus $U$ is regular. Hence it is also integrally closed. As $S$ is finite integral over $U$, we get \[
\mathcal{O}\left[z_{1}, \ldots, z_{n}\right]^{G}=\mathcal{O}\left[f_{1}, \ldots, f_{r}, z_{r+1}, \ldots, z_{n}\right]
\]

Now suppose $G$ is a subgroup of $E^*_D(r)$. Then by Lemma \ref{isom lemma} (2) and by Lemma \ref{comm alg lemma}, we may choose subgroups $A_1,\dots, A_r\subset D$ so that inside $\GL(K^n)$ we have $G=\{I + \alpha_1 E_{n,1}+\cdots +\alpha_r E_{n,r}=(S_1(\alpha_1))^T\cdots (S_r(\alpha_r))^T \mid \alpha_i \in A_i, i = 1, \ldots, r\}$. By Lemma \ref{new lemma} (2), $G$ is the same set inside $\GL(\mathcal{O}^n)$. Let $\{z_1, z_2,\dots, z_n\}$ be dual $\mathcal{O}$-basis of ${(\mathcal{O}^n)}^*$ corresponding to the $\mathcal{O}$-basis of $\mathcal{O}^n$ constructed in Lemma \ref{new lemma} (2). With respect to this basis, the action of $G$ is given by
\[
I + \alpha_1 E_{n,1}+\cdots +\alpha_r E_{n,r}\left(z_{j}\right)= \begin{cases}z_{j} & j \neq n \\ z_{n}+\alpha_1 z_1+\cdots +\alpha_r z_r  & j=n .\end{cases}
\]

Let
\[
f=\prod_{\alpha_1\in A_1,\dots, \alpha_r \in A_r}\left(z_{n}+\alpha_r z_r + \cdots +\alpha_1 z_1\right).\] Let $R=\mathcal{O}[z_1,\dots,z_n]$, $S=\mathcal{O}[z_1,\dots,z_n]^{G}$ and $U=\mathcal{O}[z_1,\dots,z_{n-1},f]$. Hence $f \in S$ and $f$ is monic in $z_n$. Now $R/(z_1,\dots,z_{n-1},f)\cong \mathcal{O}[z_n]/(z_n^{|A_1|\cdots|A_r|})$ is finitely generated $<\bar{1},\bar{u_1},\dots, \bar{u_s}>$ as an $\mathcal{O}$-module where $deg(u_i)\geq 1$. We claim that $R$ is finitely generated over $U$ as $<1, u_1,\dots,u_s>$. Let $N=<1, u_1,\dots,u_s>$. It is enough to show that $N_m=R_m$ for all $m\geq 0$. We prove this by induction. Clearly it holds for $m=0$. Assume by induction for all $d< m$. Let $\alpha\in R_m$. Hence $\alpha=\Sigma u_j \xi_j+ \delta$ where $\delta\in (z_1,\dots,z_{n-1},f)R$. So $\alpha= \Sigma u_j\xi_j +\Sigma z_k \delta'_k + f\delta''$. Now all $\delta'_k\in R_d$ and $\delta''\in R_{d'}$ for $d,d'<m$. Thus by induction $\delta'_k\in N_d$ and $\delta''\in N_{d'}$. Hence $\alpha\in N_n$ which proves our claim. Let $K'$, $Q$ and $T$ be respective field of fractions of $R$, $S$ and $U$. Now $K'/Q$ is Galois with Galois group $G$. Also $[K:T]\leq deg(f)=|A_1|\cdots |A_r|=|G|$. Thus $Q=T$. Now $R$ is finite integral over $U$. Thus $R/\pi R\cong \kappa [z_1,\dots,z_n]$   is finite integral over $U/\pi U\cong \kappa [z_1,\dots,z_{n-1},f]$ where $\kappa =\mathcal{O}/\pi \mathcal{O}$. Hence $U/\pi U$ has dimension $n$, thus is regular. Let $\mathfrak{M}$ be the maximal homogenous ideal of $U$. Now $\pi\not \in \mathfrak{M}^2$. As $U/\pi U$ is regular, thus $U$ is regular. Hence it is also integrally closed. As $S$ is finite integral over $U$, we get \[
\mathcal{O}\left[z_{1}, \ldots, z_{n}\right]^{G}=\mathcal{O}\left[z_{1}, \ldots, z_{n-1}, f\right]
\]\end{proof}

\subsection{Theorem \ref{main trasvec} for Dedekind Domains}

\begin{proof}[Proof of Theorem \ref{main trasvec}]
Now $A$ is a Dedekind domain and $K$ is its field of fractions. Let \\
$R=A [X_1, X_2,\dots , X_n]$. We need to show that $R^G$ is regular. It suffices to show that $(R^G)_{\mathfrak P}$ is regular for all homogeneous prime ideals $\mathfrak{P}$ of $R^G$. Let $S'=R^G\backslash \mathfrak P$ and $\mathfrak{p}=A\cap \mathfrak P$ and $S=A\backslash \mathfrak p$. Now $S=A\cap S'\subset S'\subset R^G\subset R$. Thus $(R^G)_{\mathfrak P}=S'^{-1}(R^G)=S'^{-1}(S^{-1}(R^G))$. As localization of regular rings is regular, it suffices to show that $S^{-1}(R^G)$ is regular. Now $S^{-1}(R^G)\subset S^{-1} R$ and hence $S^{-1}(R^G)\subset (S^{-1} R)^G$. Now let $\frac{a}{s}\in S^{-1}R$ such that $\sigma(\frac{a}{s})=\frac{a}{s}$ for all $\sigma\in G$. Thus $\frac{\sigma(a)}{s}=\frac{a}{s}$ and so there is a $t\in S$ such that $ts(\sigma(a)-a)=0$. Since $R$ is an integral domain, we have $a\in R^G$. So $S^{-1}(R^G)= (S^{-1} R)^G$. Thus it suffices to show that $(S^{-1} R)^G$ is regular.
\smallskip

Now $\mathfrak{p}$ is a prime ideal of $A$. If $\mathfrak p=(0)$, then $(S^{-1}R)^G=(K[X_1,\dots, X_n])^G$. Also as $G$ is a subgroup of $GL(A^n)$, it is naturally a subgroup of $GL(K^n)$. Now $(S^{-1}R)^G$ is regular by Theorem \ref{main for field}. If $\mathfrak p\neq (0)$, then $(S^{-1}R)^G=(A_{\mathfrak{p}}[X_1, X_2,\dots , X_n])^{G}$ where $(A_{\mathfrak{p}},\mathfrak{p} A_{\mathfrak{p}})$ is a DVR with $K$ as the field of fractions. Also as $G$ is a subgroup of $GL(A^n)$, it is naturally a subgroup of $GL(A_{\mathfrak{p}}^n)$. Now $(S^{-1}R)^G$ is regular by Theorem \ref{main for DVR}.\end{proof}




\section{Conflict of interest statement}
The authors have no conflict of interest to declare that are relevant to this article.

\section{Data availability statement}

\emph{The manuscript has no associated data.}


\bibliographystyle{plain}

\end{document}